\documentclass[a4paper,12pt]{article}
\usepackage[utf8x]{inputenc}
\usepackage{amssymb}
\usepackage{amsmath}
\usepackage{amsthm}
\usepackage{verbatim}
\usepackage{tikz}
\usetikzlibrary{arrows}
\allowdisplaybreaks

\newtheorem{thm}{Theorem}
\newtheorem{cor}[thm]{Corollary}
\newtheorem{lem}[thm]{Lemma}
\newtheorem{prop}[thm]{Proposition}
\newtheorem{example}[thm]{Example}

\usepackage{amsfonts}
\usepackage{mathtools}
\usepackage{fullpage}
\usepackage{cancel}
\usepackage{color}

\newcommand{\cC}{\mathcal{C}}
\newcommand{\cF}{\mathcal{F}}

\newcommand{\cH}{\mathcal{H}}
\newcommand{\cL}{\mathcal{L}}
\newcommand{\cM}{\mathcal{M}}

\newcommand{\cU}{\mathcal{U}}

\newcommand{\N}{\mathbb{N}}
\newcommand{\bP}{\mathbf{P}}
\newcommand{\bE}{\mathbf{E}}
\newcommand{\Sum}{\displaystyle\sum}
\newcommand{\Int}{\displaystyle\int}
\newcommand{\dt}{\mathrm{d}\tau}

\title{Codimension formulae for the intersection of fractal subsets of Cantor spaces}
\author{Casey Donoven and Kenneth Falconer}

\begin{document}
\maketitle
\begin{abstract}
We examine the dimensions of the intersection of a subset $E$ of an $m$-ary Cantor space $\cC^m$ 
with the image of a subset $F$ under a random isometry with respect to a natural metric.
We obtain almost sure upper bounds for the Hausdorff and upper box-counting dimensions of the intersection, and a lower bound for the essential supremum of the Hausdorff dimension. The dimensions of the intersections are typically  $\max\{\dim E +\dim F -\dim \cC^m, 0\}$, akin to other codimension theorems.  The upper estimates come from the expected sizes of coverings, whilst the lower
estimate is more intricate, using martingales to define a random measure on the intersection to facilitate a potential theoretic argument.
\end{abstract}

\begin{section}{Introduction}
The classical codimension formula describes the dimension of the intersection of two manifolds embedded in 
$\mathbb{R}^n$.
More specifically,  for  manifolds $E$ and  $F$, 
the dimension of $E\cap\sigma(F)$, where $\sigma$ is a 
rigid motion in $\mathbb{R}^n$, is `often' given by 
\begin{equation}\label{codim}
\dim(E\cap \sigma(F)) =\max\{\dim E + \dim F -n,0\}
\end{equation} and `typically' no more than this value. `Often' and `typical' can
be made precise in terms of a natural measure on the group of rigid motions on $\mathbb{R}^n$.
Dimension formulae for the intersection of one set with what may be regarded as a random image of another have been developed for
fractal sets,  for various definitions of fractional dimension and for other groups of transformations of $\mathbb{R}^n$. 
In particular, Mattila \cite{MatPaper,MatBook,Mat2} obtained fractal codimension formulae in the case of similarities and,
under certain restrictions,
for isometries, and Kahane \cite{Kahane} for a general class of groups which includes similarites. These formulae
have the common pattern of \eqref{codim}.

This paper presents formulae of this type for isometries under a suitable metric of the
 $m$-{\it ary Cantor space},  ${\mathcal C}^m$, defined as the set of infinite words or sequences formed from the symbols $\{1,2,\ldots,m\}$; thus  ${\mathcal C}^m = \{1,2,\ldots,m\}^\mathbb{N}$. We write $x = x_1x_2\ldots$ for a typical member of ${\mathcal C}^m$. We fix $r\in (0,1)$ and define a metric $d$ on  ${\mathcal C}^m$
 by
$$d(x_1x_2\ldots , \ y_1y_2\ldots) = r^k, \mbox{ where $k+1$ is the least integer such that $x_k\neq y_k$};$$
then $d$ is an ultrametric which induces the usual topology on the Cantor space.

Although our calculations are entirely in Cantor space, there is a visual geometric interpretation if  $r\in (0,1/m)$ when the Cantor space ${\mathcal C}^m$ may be identified with the $m$-ary Cantor set $C^m$ as a subset of the real numbers. This may be constructed in an analogous way to the usual middle-third Cantor set, starting with the unit interval and repeatedly replacing each interval by $m$ equally spaced closed subintervals of length ratio $r$ to that of the parent interval and with the end two intervals abutting the ends of the parent interval, see Figure 1.
The identification map $\phi:{\mathcal C}^m \to C^m$ is given by 
$\phi(x_1x_2\ldots)= (r+g) \sum_{i=1}^\infty (x_i - 1)r^{i-1}$ where $g$ is the gap length between two intervals of the first level of the Cantor set construction. With this identification   
the metric $d$ on ${\mathcal C}^m$  is equivalent to the Euclidean metric restricted to subsets of $C^m$. In particular, the Hausdorff and  box-counting dimensions of any subset of ${\mathcal C}^m$ defined using the metric $d$ equal the corresponding dimensions with respect to the Euclidean metric on $C^m\subset \mathbb{R}$.

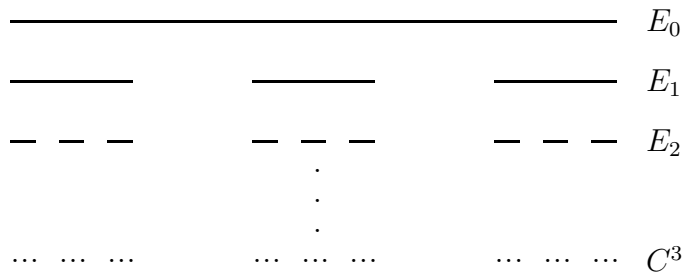
\begin{figure}[h]\label{fig:ternaryset}
  \centering
  \setlength{\unitlength}{\textwidth/10}
  \begin{picture}(6,2.5)
  \put(0.5,2.25){\line(1,0){5}}
  \put(5.75,2.17){$E_0$}
  \put(0.5,1.75){\line(1,0){1}}
  \put(2.5,1.75){\line(1,0){1}}
  \put(4.5,1.75){\line(1,0){1}}
  \put(5.75,1.67){$E_1$}
  \put(0.5,1.25){\line(1,0){.2}}
  \put(0.9,1.25){\line(1,0){.2}}
  \put(1.3,1.25){\line(1,0){.2}}
  \put(2.5,1.25){\line(1,0){.2}}
  \put(2.9,1.25){\line(1,0){.2}}
  \put(3.3,1.25){\line(1,0){.2}}
  \put(4.5,1.25){\line(1,0){.2}}
  \put(4.9,1.25){\line(1,0){.2}}
  \put(5.3,1.25){\line(1,0){.2}}
  \put(5.75,1.17){$E_2$}
  \put(3,1){.}
  \put(3,.75){.}
  \put(3,.5){.}
  \put(.5,.25){.}
  \put(.58,.25){.}
  \put(.66,.25){.}
  \put(.9,.25){.}
  \put(.98,.25){.}
  \put(1.06,.25){.}
  \put(1.3,.25){.}
  \put(1.38,.25){.}
  \put(1.46,.25){.}
  \put(2.5,.25){.}
  \put(2.58,.25){.}
  \put(2.66,.25){.}
  \put(2.9,.25){.}
  \put(2.98,.25){.}
  \put(3.06,.25){.}
  \put(3.3,.25){.}
  \put(3.38,.25){.}
  \put(3.46,.25){.}
  \put(4.5,.25){.}
  \put(4.58,.25){.}
  \put(4.66,.25){.}
  \put(4.9,.25){.}
  \put(4.98,.25){.}
  \put(5.06,.25){.}
  \put(5.3,.25){.}
  \put(5.38,.25){.}
  \put(5.46,.25){.}
  \put(5.75,.17){$C^3$}
  \end{picture}
  \caption{The  3-ary Cantor set with $r = 1/5$}
 \end{figure}
 
 Let $\mathrm{Iso }\,{\mathcal C}^m$ denote the group of isometries of ${\mathcal C}^m$. (With the Cantor set interpretation these isometries may be visualised as combinations of 
permutations of the construction intervals of the Cantor set at various levels.) The group $\mathrm{Iso }\,{\mathcal C}^m$ and its subgroups continue to be studied intensively, both from group theoretic and dynamical viewpoints, see for example \cite{BGN}. This paper provides further insight into the geometry of the group. There is a natural invariant probability measure  ${\mathbf P}$ on 
$\mathrm{Iso }\,{\mathcal C}^m$ such that the isometries that induce each admissible permutation of the construction intervals of $C^m$ at a given level have equal probability, see below.

We bring together our main results in the following statement, where $\dim_H$, $\dim_B$ and $\overline{\dim}_B$ denote 
Hausdorff, box-counting and upper box-counting dimension respectively, see \cite{FractalGeometry} for definitions. Note that
$\dim_H  {\mathcal C}^m= \dim_B  {\mathcal C}^m = -\log m /\log r$.

\begin{thm}
Let $E,F \subset {\mathcal C}^m$ be Borel sets. Then for a random isometry $\sigma \in \mathrm{Iso }{\mathcal C}^m$:
\begin{itemize}
\item[$(i)$] almost surely \  $\overline{\dim}_B(E\cap \sigma(F)) \leq\max\big\{\overline{\dim}_B E + \overline{\dim}_B F +\log m /\log r, \ 0 \big\},$

\item[$(ii)$] almost surely \  $\dim_H(E\cap \sigma(F)) \leq\max\big\{\dim_H E + \overline{\dim}_B F +\log m /\log r, \ 0 \big\},$

\item[$(iii)$] $\mathrm{esssup}_{\sigma\in \mathrm{Iso }{\mathcal C}^m}\big\{\dim_H(E\cap \sigma(F))\} \geq
\max\big\{\dim_H E + \dim_H F +\log m /\log r, \ 0 \big\}.$
\end{itemize}
\end{thm}

Parts (i) and (ii) will be obtained using covering arguments. The lower bound (iii) is more complicated, and uses measures 
defined on $E$ and $F$ to set up a measure martingale that converges almost surely to a measure supported on $E\cap \sigma(F)$. A potential-theoretic argument then gives lower bounds for the dimension. 

Some basic notation will be used throughout the paper. For each $k$ and each finite word $x_1 x_2 \ldots x_k \in \{1,2,\ldots,m\}^k$ we associate the {\it level-k cylinder} $\{x_1 x_2 \ldots x_k y_{k+1}  y_{k+2}\ldots : 1\leq y_i \leq m\}$  which we will generally  refer to as an  {\it interval} $I$, to correspond to the Cantor set interpretation. We write $U_k$ for the set of all $k$th level intervals. Also,  for $A \subset {\mathcal C}^m$ we use $U_k(A)$ to denote the set of $k$th level intervals 
that intersect $A$ non-trivially, specifically $U_k(A) = \{I \in U_k : I \cap A \neq \emptyset\}$, so that the intervals of  
$U_k(A)$ form a cover of $A$ for each $k$. We will write $|\cdot|$ to denote cardinality, so in particular 
$|U_k(A)|$ is the number of level $k$ intervals that intersect $A$. We write $d(A) = \inf\{d(x,y) :x,y \in A\}$ for the {\it diameter} of a (non-empty) set $A \subset {\mathcal C}^m$, so that $d(I) = r^{-k}$ if $I$ is a $k$th-level interval.

A convenient way of characterising the isometries $\mathrm{Iso }\,{\mathcal C}^m$ is using the natural correspondence of $\cC$ with the infinite rooted $m$-ary tree, ${\mathcal T}^m$. The boundary of ${\mathcal T}^m$ is identified with the Cantor space $\cC^m$ and the vertices correspond to the intervals or cylinders. Then the  group of graph automorphisms of the rooted tree ${\mathcal T}^m$ correspond to the group of isometries $\mathrm{Iso }\,{\mathcal C}^m$ of the Cantor space. An automorphism acts by 
`twisting' the tree at sets of nodes, perhaps infinitely
many, rearranging the children of each node into a new permutation, see Figure 2.

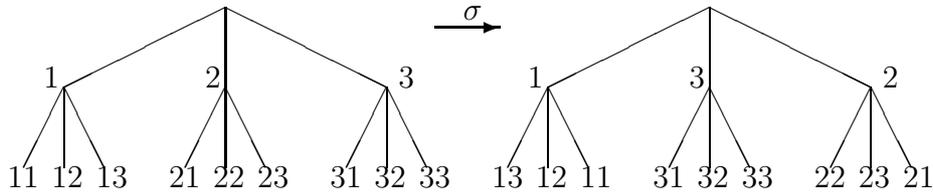
\begin{figure}[h]\label{fig:autoftree}
\centering 
  \setlength{\unitlength}{\textwidth/15}
  \begin{picture}(12,2.5)
  \put(3,2.25){\line(0,-1){1}}
  \put(3,2.25){\line(2,-1){2}}
  \put(3,2.25){\line(-2,-1){2}}
  \put(1,1.25){\line(0,-1){1}}
  \put(1,1.25){\line(1,-2){.5}}
  \put(1,1.25){\line(-1,-2){.5}}
  \put(3,1.25){\line(0,-1){1}}
  \put(3,1.25){\line(1,-2){.5}}
  \put(3,1.25){\line(-1,-2){.5}}
  \put(5,1.25){\line(0,-1){1}}
  \put(5,1.25){\line(1,-2){.5}}
  \put(5,1.25){\line(-1,-2){.5}}
  \put(9,2.25){\line(0,-1){1}}
  \put(9,2.25){\line(2,-1){2}}
  \put(9,2.25){\line(-2,-1){2}}
  \put(7,1.25){\line(0,-1){1}}
  \put(7,1.25){\line(1,-2){.5}}
  \put(7,1.25){\line(-1,-2){.5}}
  \put(9,1.25){\line(0,-1){1}}
  \put(9,1.25){\line(1,-2){.5}}
  \put(9,1.25){\line(-1,-2){.5}}
  \put(11,1.25){\line(0,-1){1}}
  \put(11,1.25){\line(1,-2){.5}}
  \put(11,1.25){\line(-1,-2){.5}}
  \put(.75,1.25){1}
  \put(2.75,1.25){2}
  \put(5.15,1.25){3}
  \put(.30,0){11}
  \put(.85,0){12}
  \put(1.40,0){13}
  \put(2.30,0){21}
  \put(2.85,0){22}
  \put(3.40,0){23}
  \put(4.30,0){31}
  \put(4.85,0){32}
  \put(5.40,0){33}
  \put(6.75,1.25){1}
  \put(8.75,1.25){3}
  \put(11.15,1.25){2}
  \put(6.30,0){13}
  \put(6.85,0){12}
  \put(7.40,0){11}
  \put(8.30,0){31}
  \put(8.85,0){32}
  \put(9.40,0){33}
  \put(10.30,0){22}
  \put(10.85,0){23}
  \put(11.40,0){21}
  \thicklines
  \put(5.6,2){\vector(1,0){.8}}
  \put(5.95,2.1){$\sigma$}
  \end{picture}
  \caption{An automorphism $\sigma$ acting on two levels of the ternary Cantor space}
\end{figure}

The natural invariant probability space $(\mathrm{Iso }\,{\mathcal C}^m, {\mathcal F},\mathbf{P})$ on the isometries of $\cC^m$ is defined 
 as follows. For each $k$ let $\pi$ be an admissible 
permutation of the intervals of $U_k$ (i.e. one that is achievable by some $\sigma \in \mathrm{Iso }\,{\mathcal C}^m$) and let ${\mathcal I}_\pi$ be the set of all isometries $\sigma \in \mathrm{Iso }\,{\mathcal C}^m$ 
such that $\sigma(I) =  \pi (I)$ for all $I\in U_k$. Let ${\mathcal F}_k$ be the finite sigma-field consisting of 
finite unions of all such ${\mathcal I}_\pi$. We define a probability on 
${\mathcal F}_k$ by ascribing equal probability to each ${\mathcal I}_\pi$, so that
$\mathbf{P} ({\mathcal I}_\pi)  = m^{-k(k+1)/2}$, and extending to ${\mathcal F}_k$. These sigma-fields form an 
increasing sequence and we define ${\mathcal F}= {\mathcal S}(\bigcup_{k=0}^\infty{\mathcal F}_k)$ for the sigma-field generated by their union and extend $\mathbf{P}$ to 
${\mathcal F}$ in the usual way. Note that for $I,J\in U_k$ and $\sigma\in \mathrm{Iso}\, \cC^m$,
$\bP\big(\sigma(I)=J\big)=m^{-k}$.
\end{section}

%%%%%%%%%%%%%%%%%%%%%%%%%
%%%%%%%%%%%%%%%%%%%%%%%%%

\begin{section}{Upper Box Counting Dimension: Upper Bound}
In this section, we bound the upper box counting dimension of the intersection of a subset of $\cC^m$ with a random image of another subset.

\begin{thm}\label{thm2}
Let $E, F \subset \cC^m$. Then, almost surely, 
\begin{equation}\label{upperboxbound}
\overline{\dim}_B(E\cap \sigma (F))
\leq \max\bigg\{\overline{\dim}_B E+\overline{\dim}_B F+\frac{\log m }{\log r },0\bigg\}.
\end{equation}
\end{thm}

\begin{proof}
First note that
\[U_k\big(E\cap\sigma(F)\big)\subset U_k\big(E\big)\cap U_k\big(\sigma(F)\big)=U_k\big(E\big)\cap \sigma\big(U_k(F)\big)\]
For $k \geq 0$ and  $J\in U_k(F)$, consider the indicator function $\chi_J: \mathrm{Iso }\,{\mathcal C}^m \to \{0,1\}$ such that $\chi_J(\sigma)= 1$ when $\sigma(J)\in U_k(E)$. Then
\begin{equation*}
\big|U_k\big(E\cap\sigma(F)\big)\big|\leq |U_k\big(E\big)\cap \sigma\big(U_k(F)\big)|=\Sum_{J\in U_k(F)}\chi_J(\sigma).
\end{equation*}
A random automorphism $\sigma$ takes an interval $J\in U_k$ to a particular interval $I\in U_k$ with probability $m^{-k}$, therefore
for all $J\in U_k$
\begin{eqnarray*}
\bE(\chi_J(\sigma))=m^{-k}|U_k(E)|.
\end{eqnarray*}
This implies
\begin{eqnarray*}
\bE\big(\big|U_k\big(E\cap\sigma(F)\big)\big|\big)
\leq\Sum_{J\in U_k(F)}\bE\big(\chi_J(\sigma)\big)
=m^{-k}|U_k(E)||U_k(F)|.
\end{eqnarray*}

Assume that  $\overline{\dim}_B E + \overline{\dim}_B F +\log m /\log r> 0$, otherwise there is nothing to prove.   Take $\alpha$ and $\beta$ such that  $\alpha>\overline{\dim}_B E$ and $\beta>\overline{\dim}_B F$. From the definition of upper box dimension,
 there exist $c_1, c_2>0$ such that, for all $k\geq 0$,
\begin{eqnarray*}
|U_k(E)|\leq  c_1r^{-k\alpha}\quad \mbox{and}  \quad |U_k(F)|\leq  c_2r^{-k\beta}. \label{eqn:upperboxbound}
\end{eqnarray*}
Setting $c=c_1c_2$,  for all $k>0$,
\begin{eqnarray*}
\bE\Big(\big|U_k\big(E\cap\sigma(F)\big)\big|\Big)
&\leq&cr^{-k(\alpha+\beta+\log m /\log r)}
= cr^{-kd},
\end{eqnarray*}
where $d=\alpha+\beta+\log m /\log r>0$. Let $\epsilon>0$. Then
\begin{eqnarray*}
\bE\Big(\Sum_{k=0}^\infty r^{k(d+\epsilon)}\big|U_k\big(E\cap\sigma(F)\big)\big|\Big)
\leq \Sum_{k=0}^\infty cr^{k \epsilon}<\infty.
\end{eqnarray*}
Thus, almost surely, there exists a random $C<\infty$ such that
\begin{eqnarray*}
\Sum_{k=0}^\infty r^{k(d+\epsilon)}\big|U_k\big(E\cap\sigma(F)\big)\big|&\leq& C,
\end{eqnarray*}
so
\begin{eqnarray*}
\big|U_k\big(E\cap\sigma(F)\big)\big|&\leq& Cr^{-k(d+\epsilon)}
\end{eqnarray*}
for all $k \geq 0$.
When calculating upper box dimension it is enough to consider coverings by intervals of lengths $r^{-k}$ for $0 \leq k<\infty$, so
 \begin{eqnarray*}
\overline{\dim}_B(E\cap \sigma(F))
\leq d+\epsilon
= \alpha+\beta+\log m /\log r+\epsilon.
\end{eqnarray*}
Taking $\epsilon$  arbitrarily small and $\alpha$ and $\beta$ arbitrarily close to  $\overline{\dim}_B E$ and $\overline{\dim}_B F$ gives \eqref{upperboxbound}.
\end{proof}

Note that a minor variation on this argument shows that $E\cap \sigma(F) = \emptyset$ almost surely if 
$\overline{\dim}_B E+\overline{\dim}_B F+\log m /\log r <0$.

\end{section}

\begin{section}{Hausdorff Dimension: Upper Bound}
We will now obtain an upper bound for the Hausdorff dimension of the intersections. 
We write ${\mathcal H}^s$ for $s$-dimensional Hausdorff measure, see \cite{MatBook} for its definition and properties.
However, rather than work directly with Hausdorff measures, it is convenient to use an equivalent definition based on coverings of subsets of $\cC^m$ by intervals or cylinders rather than by arbitrary sets.  Let $\cU=\bigcup_{k=0}^\infty U_k$ denote the collection of intervals and let $d(\cdot)$ denote the diameter of a set with respect to the metric $d(\cdot,\cdot)$. For  $s\geq 0$,  $\delta >0$ and $A\subset C^m$,  define the $\delta$-premeasures by 
\begin{eqnarray*}
\cM_{\delta}^s(A)&=&\inf\Big\{\Sum_{i=1}^\infty d(I_i)^{s}:A \subset \bigcup_{i=1}^\infty   I_i,d(I_i) \leq \delta\Big\}
\end{eqnarray*}
and let
\begin{eqnarray*}
\cM^s(A)&=&\lim_{\delta\to0}\cM_{\delta}^s(A).
\end{eqnarray*}
Then  $\cM^s$ is a Borel measure on $\cC^m$. 
\begin{lem}\label{equivmes}
For all $A \subset \cC^m$,  $\cM^s(A)=\cH^s(A)$. In particular,
$\dim_H(A)=\sup\{s:\cM^s(A)>0\} = \inf\{s:\cM^s(A)=0\}$.
\end{lem}
\begin{proof}
Cleary $\cH^s(A)\leq\cM^s(A)$ for all $A$, since any admissible cover for $\cM^s$ is an admissible cover for $\cH^s$.  
For the opposite inequality, note that the diameter of any set $O \subset \cC^m$ equals that of the smallest interval $I$ of $\cU$ that contains $O$. Thus replacing any covering set $O$ by the corresponding interval $I$ does not change the diameters involved in the definitions of the measures, so $\cM^s(A)\leq\cH^s(A)$.
\end{proof}

\begin{thm}\label{hausub}
Let $E, F \subset \cC^m$. Almost surely
\begin{equation}\label{dimhausbound}
\dim_H(E\cap\sigma(F))
\leq \max\bigg\{\dim_H E+\overline{\dim}_B F+\frac{\log m}{\log r},0\bigg\}.
\end{equation}
\end{thm}
\begin{proof}
Take $\alpha$ and $\beta$ with  $\alpha>\dim_H E$ and $\beta>\overline{\dim}_B F$.
Then there exists $c>0$ such that for all $k\geq 0$
\begin{eqnarray*}
|U_k(F)|&\leq& cr^{-k\beta}.
\end{eqnarray*}
By Lemma \ref{equivmes}, for all $\delta>0$ we can find intervals
$I_i \in \cU$ such that $E\subset\bigcup_i I_i$, $d(I_i)\leq \delta$, and $\Sum_i d(I_i)^\alpha\leq 1$.  Taking only those intervals
 $I_i$ that intersect $\sigma(F)$ non-trivially, gives a $\delta$-cover of $E\cap\sigma(F)$ and therefore, for $s>0$,
\begin{eqnarray*}
\cM^s_{\delta}(E\cap\sigma(F))&\leq&\Sum_i\{d(I_i)^s:\sigma^{-1}(I_i)\cap F\neq \emptyset\}.
\end{eqnarray*}
Taking the expectation,
\begin{eqnarray*}
\bE\big(\cM^s_{\delta}(E\cap\sigma(F))\big)&\leq&\Sum_id(I_i)^s\bP(\sigma^{-1}(I_i)\cap F\neq \emptyset).
\end{eqnarray*}
If $I_i\in U_k$, then $d(I_i)=r^k$, so 
\begin{eqnarray*}
\bP(\sigma^{-1}(I_i)\cap F\neq \emptyset)&=&m^{-k}|U_k(F)|\\
&\leq&cm^{-k}r^{-k\beta}\\
&=&cd(I_i)^{-(\beta+\log m/\log r)}.
\end{eqnarray*}
Thus
\begin{eqnarray*}
\bE\big(\cM^s_{\delta}(E\cap\sigma(F))\big)&\leq&c\Sum_id(I_i)^{s-(\beta+\log m/\log r)}\\
&=&c\Sum_i d(I_i)^{\alpha}d(I_i)^{s-(\alpha+\beta+\log m/\log r)}\\
&\leq&c\Sum_id(I_i)^{\alpha}\delta^{s-(\alpha+\beta+\log m/\log r)}\\
&\leq&c\delta^{s-(\alpha+\beta+\log m/\log r)}
\end{eqnarray*}
provided that $s-(\alpha+\beta+\log m/\log r) >0$.
Taking $\delta=2^{-k }$ and summing,
\begin{eqnarray*}
\bE\Big(\Sum_{k=1}^\infty\cM_{2^{-k}}^{s}(E\cap\sigma(F))\Big)
&\leq&c\Sum_{k=1}^\infty2^{-k(s-(\alpha+\beta+\log m/\log r))}<\infty.
\end{eqnarray*}
This implies that, almost surely,
\begin{eqnarray*}
\Sum_{k=1}^\infty\cM_{2^{-k}}^{s}(E\cap\sigma(F))&<&\infty
\end{eqnarray*}
so
\begin{eqnarray*}
\cM^{s}(E\cap\sigma(F))= \lim_{\delta \to 0} \cM_{\delta}^{s}(E\cap\sigma(F))=0.
\end{eqnarray*}
In particular, by Lemma \ref{equivmes},  $\dim_H(E\cap\sigma(F))\leq s$ almost surely, provided that
$ s >\alpha+\beta+\log m/\log r$.
This holds for  $\alpha$ and $\beta$ arbitrarily close to  $\dim_H E$ and $\overline{\dim}_B F$, giving \eqref{dimhausbound}.
\end{proof}

Again, minor changes to the argument show that $E\cap \sigma(F) = \emptyset$ almost surely if 
$\dim_H E+\overline{\dim}_B F+\log m /\log r <0$.

Note that if, as often happens, either $E$ or $F$ is sufficiently regular to have equal Hausdorff and upper box dimensions, then we get $\dim_H$ throughout inequality \eqref{dimhausbound}.
\end{section}

%%%%%%%%%%%%%%%%%%%%%%%%%
%%%%%%%%%%%%%%%%%%%%%%%%%

\begin{section}{Hausdorff Dimension: Lower Bound}
In this section we obtain a lower bound for the essential supremum of $\dim_H (E\cap\sigma(F))$ where    $\sigma$ is a random isometry. To achieve this we put Frostman-type measures on $E$ and $F$ and define a measure martingale that converges to a measure on $E\cap\sigma(F)$. By examining the $s$-energy of this measure we obtain a lower bound for the dimension that occurs with positive probability. The bulk of the calculation is devoted to showing that the martingales are $\cL^2$-bounded.

Throughout this section, $E,F$ will be Borel subsets of $\cC^m$ and $0< \alpha <\dim_H E$ and $0< \beta <\dim_H F$. Eventually we will take $\alpha$ and $\beta$ arbitrarily close to the respective dimensions.
\begin         {lem}\label{lem:measurebounds}
There exist probability measures $\mu$ and $\nu$, 
with compact support contained in $E$ and $F$ respectively, and positive constants $c_E$ and $c_F$ such that for all $k\geq0$ and $I\in U_k$,
\begin{equation}\label{frost}
\mu(I)\leq c_Er^{k\alpha} \quad \text{and}\quad \nu(I)\leq c_Fr^{k\beta}.
\end{equation}
\end{lem}

\begin{proof}
By Frostman's Lemma for metric spaces \cite{MatBook,Rogers},  there are probability
measures $\mu$ and $\nu$, such that $\mu(A)\leq c_E d(A)^{\alpha}$ and
$\nu(A)\leq c_F d(A)^{\beta}$ for all $A \subset \cC^m$. If $I\in U_k$, then  $d(I) = r^k$ so the conclusion follows.
\end{proof}

Let $k\in\N$ and let $\mu$ and $\nu$ be given by Lemma \ref{lem:measurebounds}. For all $A\in U_k$ and $l\geq k$ define a random variable
\begin{eqnarray}\label{martdef}
\tau_l(A)&=&m^l\Sum_{I\in U_l(A)}\mu(I)\nu(\sigma^{-1}(I)).
\end{eqnarray}
Note that $\tau_l(A)$ is $\cF_l$ measurable, where  $\cF_l$ is the sigma-field generated by the isometries defined at the $l$th level, see Section 1. 
We will show that $\{\tau_l(A), \cF_l\}_{l\geq k}$ is an $\cL^2$-bounded martingale and that the limits of these martingales give rise to an additive set function on ${\mathcal U}= \bigcup_{k=0}^\infty U_k$ and thus a measure on  $\cC^m$.

\begin{lem}
Let $A\in U_k$. Then $\{\tau_l(A), \cF_l\}_{l\geq k}$ is a non-negative martingale.  
\end{lem}

\begin{proof}
Let  $l\geq k+1$. For each  $I\in U_l$, we write $I'\in U_{l-1}$ for the parent interval of $I$. Then
\begin{eqnarray}\label{martsum}
\bE\big(\tau_l(A)|\cF_{l-1}\big)&=&m^l\Sum_{I\in U_l(A)}\mu(I)\bE\big(\nu(\sigma^{-1}(I))|\cF_{l-1}\big).
\end{eqnarray}
Conditional on $\cF_{l-1}$, $\sigma^{-1}(I)$ is equally likely to be any of the $m$ children of $\sigma^{-1}(I')$ so
\begin{eqnarray*}
\bE\big(\nu(\sigma^{-1}(I))|\cF_{l-1}\big)&=&m^{-1}\nu(\sigma^{-1}(I')).
\end{eqnarray*}
Partitioning the sum \eqref{martsum} over the intervals $I'$ at the $(l-1)$th level gives

\begin{eqnarray*}
m^l\Sum_{I'\in U_{l-1}(A)}\Sum_{\substack{I\subset I'\\I\in U_l}}\mu(I)\bE\big(\nu(\sigma^{-1}(I))|\cF_{l-1}\big)&=&
m^l\Sum_{I'\in U_{l-1}(A)}\Sum_{\substack{I\subset I'\\I\in U_l}}m^{-1}\mu(I)\nu(\sigma^{-1}(I'))\\
&=&m^{l-1}\Sum_{I'\in U_{l-1}(A)}\mu(I')\nu(\sigma^{-1}(I'))\\
&=&\tau_{l-1}(A).
\end{eqnarray*}
Clearly $\tau_l(A)\geq0$ for all $l$, so  $\{\tau_l(A), \cF_l\}_{l\geq k}$ is a non-negative martingale.
\end{proof}

In proving  $\cL^2$-boundedness, we will need the following inequality.

\begin{lem}\label{lem:Youngs}
Let $ x_1,x_2,\ldots,x_m\geq 0$ be real numbers. Then
\begin{eqnarray}\label{ineq}
m \Sum_{i\neq j} x_ix_j &\leq& (m-1) \Sum_{i,j} x_ix_j 
\end{eqnarray}
\end{lem}
\begin{proof}
Young's Inequality implies that $x_ix_j \leq \frac{1}{2} x_i^2+ \frac{1}{2}x_j^2$ for each pair $i$ and $j$.  By summing over all pairs  such that $i\neq j$, we see that
\begin{eqnarray*}
\Sum_{i\neq j} x_ix_j &\leq& (m-1)\Sum_{i} x_i^2
\end{eqnarray*}
and therefore
\begin{eqnarray*}
m \Sum_{i\neq j} x_ix_j &=& \Sum_{i\neq j} x_ix_j + (m-1) \Sum_{i\neq j} x_ix_j\\
&\leq& (m-1) \Sum_{i=1}^m x_i^2 + (m-1) \Sum_{i\neq j} x_ix_j\\
&=& (m-1)\Sum_{i,j}x_ix_j.
\end{eqnarray*}
\end{proof}
\begin{lem}\label{lem:boundedL2}
Assume that $\alpha+\beta>-\log m /\log r $. There is a constant $c_0$ such that for all $A \in U_k$ and $l \geq k$,
\begin{equation}\label{squarebound}
\bE\big(\tau_l(A)^2\big)\leq c_0\mu(A) r^{k(\alpha+\beta+\log m / \log r )}.
\end{equation}
In particular, the martingale $\{\tau_l(A),\cF_{l}\}_{l\geq k}$ is $\cL^2$-bounded. \end{lem}

\begin{proof}
Let $A \in U_k$. We will first bound $\bE\big(\tau_l(A)^2|\cF_{l-1}\big)$ in terms of $\tau_{l-1}(A)$  where $l \geq k+1$, to obtain \eqref{start} below. As before, we make the convention that $I'\in U_{l-1}$ is the parent interval of $I\in U_l$.

The expectation of $\tau_l(A)^2$ conditional on $\cF_{l-1}$ breaks down into three sums:
\begin{eqnarray}
\bE\big(\tau_l(A)^2|\cF_{l-1}\big)&=&m^{2l}\Sum_{I,J\in U_l(A)}\mu(I)\mu(J)\bE\Big(\nu(\sigma^{-1}(I))\nu(\sigma^{-1}(J))\big|\cF_{l-1}\Big)\nonumber\\
&=&m^{2l}\Sum_{\substack{I',J'\in U_{l-1}(A)\\I'\neq J'}}\Sum_{\substack{I\subset I'\\J\subset J'}}
\mu(I)\mu(J)\bE\Big(\nu(\sigma^{-1}(I))\nu(\sigma^{-1}(J))\big|\cF_{l-1}\Big) \label{eqn:L2sum1}\\
&&+ \ m^{2l}\Sum_{I'\in U_{l-1}(A)}\Sum_{\substack{I,J\subset I'\\I\neq J}}
\mu(I)\mu(J)\bE\Big(\nu(\sigma^{-1}(I))\nu(\sigma^{-1}(J))\big|\cF_{l-1}\Big) \label{eqn:L2sum2}\\
&&+ \ m^{2l}\Sum_{I'\in U_{l-1}(A)}\Sum_{I\subset I'}
\mu(I)^2\bE\Big(\nu(\sigma^{-1}(I))^2\big|\cF_{l-1}\Big). \label{eqn:L2sum3}
\end{eqnarray}
We estimate the expectation term in (\ref{eqn:L2sum1}), (\ref{eqn:L2sum2}), and
(\ref{eqn:L2sum3}) separately.\\
\\
{\it Case 1}: The sum in (\ref{eqn:L2sum1}) is over intervals $I,J\in U_l$ with different parent intervals,
$I',J'\in U_{l-1}$ respectively.
This affords independence in the calculation of conditional expectation, so
\begin{eqnarray*}
\bE\Big(\nu(\sigma^{-1}(I))\nu(\sigma^{-1}(J))\big|\cF_{l-1}\Big)&=&
\bE\Big(\nu(\sigma^{-1}(I))\big|\cF_{l-1}\Big)\bE\Big(\nu(\sigma^{-1}(J))\big|\cF_{l-1}\Big).
\end{eqnarray*}
Given $\cF_{l-1}$, $\sigma^{-1}(I)$ is equally likely to be any one of the $m$ intervals $I_0\in U_{l}$ that are children of $\sigma^{-1}(I')$, so
\begin{eqnarray*}
\bE\Big(\nu(\sigma^{-1}(I))\big|\cF_{l-1}\Big)&=&\Sum_{\substack{I_0\subset \sigma^{-1}(I')\\I_0\in U_{l}}}\frac{\nu(I_0)}{m}
\ = \ \frac{\nu\big(\sigma^{-1}(I')\big)}{m},
\end{eqnarray*}
with a similar expression for the term involving $\sigma^{-1}(J)$.
The expected value in (\ref{eqn:L2sum1}) then becomes
\begin{eqnarray*}
\bE\Big(\nu(\sigma^{-1}(I))\nu(\sigma^{-1}(J))\big|\cF_{l-1}\Big)&=&\frac{\nu(\sigma^{-1}(I'))\nu(\sigma^{-1}(J'))}{m^2}.
\end{eqnarray*}
{\it Case 2}: The sum in (\ref{eqn:L2sum2}) is over two disjoint intervals with the same parent interval, $I'\in U_{l-1}$.  The pair of intervals, $\sigma^{-1}(I)$ and $\sigma^{-1}(J)$, is equally likely to be any of the $m(m-1)$ pairs of distinct
children $I_0$ and $J_0$ of  $\sigma^{-1}(I')\in U_{l-1}$, and using 
\eqref{ineq},
\begin{eqnarray*}
\bE\Big(\nu(\sigma^{-1}(I))\nu(\sigma^{-1}(J))\big|\cF_{l-1}\Big)
&=& \Sum_{\substack{I_0,J_0\subset \sigma^{-1}(I')\\I_0\neq J_0}}\nu(I_0)\nu(J_0)\frac{1}{m(m-1)}\\
&\leq& \Sum_{I_0,J_0\subset \sigma^{-1}(I')}\nu(I_0)\nu(J_0)\frac{1}{m^{2}}\\
&=&\frac{\nu(\sigma^{-1}(I'))^2}{m^2}.
\end{eqnarray*}
{\it Case 3}: The sum in (\ref{eqn:L2sum3}) is over intervals $I$ with parent interval $I'$, and $\sigma^{-1}(I)$ is equally likely to be any of the $m$ children of $\sigma^{-1}(I')$, say $I_0$. Combining this with the inequality $\nu(I_0)\leq c_F r^{l\beta}$ from
\eqref{frost},
\begin{eqnarray*}
\bE\Big(\nu(\sigma^{-1}(I))^2\big|\cF_{l-1}\Big)
&=&\Sum_{I_0\subset \sigma^{-1}(I')}\nu(I_0)^2m^{-1}\\
&\leq&\Sum_{I_0\subset \sigma^{-1}(I')}c_Fr^{l\beta}\nu(I_0)m^{-1}\\
&=&\frac{c_F r^{l\beta}\nu(\sigma^{-1}(I'))}{m}.
\end{eqnarray*}
Incorporating these three cases in \eqref{eqn:L2sum1}--\eqref{eqn:L2sum3} and using that $\mu(I)\leq c_E r^{l\alpha}$ for every $I \in U_l$,
\begin{eqnarray}
\bE\big(\tau_l(A)^2|\cF_{l-1}\big)&\leq&m^{2l}\Sum_{I',J'\in U_{l-1}(A)}\Sum_{\substack{I\subset I'\\J\subset J'}}
\mu(I)\mu(J)\frac{\nu(\sigma^{-1}(I'))\nu(\sigma^{-1}(J'))}{m^2}\nonumber\\
&&+ \ m^{2l}\Sum_{I'\in U_{l-1}(A)}\Sum_{I\subset I'}
\mu(I)c_E r^{l\alpha}\frac{c_Fr^{l\beta}\nu(\sigma^{-1}(I'))}{m}\nonumber\\
&=&m^{2(l-1)}\Sum_{I',J'\in U_{l-1}(A)}\mu(I')\mu(J')\nu(\sigma^{-1}(I'))\nu(\sigma^{-1}(J'))\nonumber\\
&&+ \ c_Ec_Fr^{l\alpha}r^{l\beta}m^{l}m^{l-1}\Sum_{I'\in U_{l-1}(A)}\mu(I')\nu(\sigma^{-1}(I'))
\nonumber\\
&=&\tau_{l-1}(A)^2+c\tau_{l-1}(A)r^{l(\alpha+\beta+\log m / \log r )},\label{start}
\end{eqnarray}
where $c=c_Ec_F$. 

We apply this inequality inductively (working backwards) to bound $\bE\big(\tau_l(A)^2|\cF_{k}\big)$ where $A \in U_k$. Assume that for some $j$ with 
$k+1\leq j\leq l-1$,
\begin{equation}\label{indhyp}
\bE\big(\tau_l(A)^2|\cF_{j}\big)
 \leq \tau_{j}(A)^2+c\tau_{j}(A)
\sum_{i=j+1}^{l}r^{i(\alpha+\beta+\log m / \log r )};
\end{equation}
when $j=l-1$ this is just \eqref{start}. Using the tower property for conditional expectation, inequalities \eqref{indhyp},  \eqref{start} (with $j$ playing the role of $l$),  and that $\tau_j$ is a martingale,
\begin{eqnarray*}
\bE\big(\tau_l(A)^2|\cF_{j-1}\big) &=& \bE\big(\bE(\tau_l(A)^2|\cF_{j})|\cF_{j-1}\big)\\
&\leq&\bE\big(\tau_{j}(A)^2|\cF_{j-1}\big)+c\bE\big(\tau_{j}(A)|\cF_{j-1}\big)
\sum_{i=j+1}^{l}r^{i(\alpha+\beta+\log m / \log r )}\\
&\leq&\tau_{j-1}(A)^2 + c\tau_{j-1}(A)
\sum_{i=j}^{l}r^{i(\alpha+\beta+\log m / \log r )},
\end{eqnarray*}
for the inductive step.
Taking $j=k$ in  \eqref{indhyp} and recalling that $\alpha+\beta+\log m / \log r>0$, we conclude that
\begin{equation}\label{geoser}
\bE\big(\tau_l(A)^2|\cF_{k}\big)\leq
\tau_{k}(A)^2 + c_1 \tau_{k}(A)r^{k(\alpha+\beta+\log m / \log r )},
\end{equation}
where $c_1$ does not depend on $l,k$ or $A$.

With $A \in U_k$ as before, we take unconditional expectations of this inequality, and use \eqref{martdef} and \eqref{frost}: 
\begin{eqnarray}
\bE\big(\tau_l(A)^2\big) &\leq& \bE\big(\tau_{k}(A)^2\big)+c_1 \bE\big(\tau_{k}(A)\big) r^{k(\alpha+\beta+\log m / \log r)}\nonumber\\
&=& m^{2k}\mu(A)^2\bE\big(\nu(\sigma^{-1}(A))^2\big) + 
c_1m^k\mu(A)\bE\big(\nu(\sigma^{-1}(A))\big) r^{k(\alpha+\beta+\log m / \log r )} \nonumber\\
&=& m^{2k}\mu(A)^2\Sum_{I\in U_k}\nu(I)^2m^{-k}
+ c_1m^k\mu(A)\Sum_{I\in U_k}\nu(I)m^{-k} r^{k(\alpha+\beta+\log m /\log r )} \nonumber\\
&\leq& c_Ec_F m^k\mu(A)\Sum_{I\in U_k}\nu(I)r^{k(\alpha+\beta)}+ c_1\mu(A)r^{k(\alpha+\beta+\log m /\log r )} \nonumber\\
&\leq&c_0\mu(A) r^{k(\alpha+\beta+\log m / \log r )}, \label{eqn:L2expectbdd}
\end{eqnarray}
where $c_0 = c_Ec_F +c_1$.
\end{proof}

We now use the $\tau_l$ to obtain a limiting measure.  First let $A\in{\mathcal S}(U_k)$, the sigma-algebra of subsets of $\cC^m$ generated by
the $k$th level intervals, so $A$ is a (finite) union of intervals in $U_k$.  For $l\geq k$ define
\begin{eqnarray*}
\tau_l(A) &=& m^l\Sum_{I\in U_l(A)}\mu(I)\nu(\sigma^{-1}(I)).
\end{eqnarray*}
Note that when $A\in U_k$ this coincides with the definition of $\tau_l(A)$  given by \eqref{martdef}.
For all  $k$ and all $A\in{\mathcal S}(U_k)$,  $\{\tau_l(A), \cF_l\}_{l\geq k}$  is a martingale as a finite sum of martingales. Thus $\tau_l(A)$ converges almost surely to a random variable on the sigma-field
 $\cF= {\mathcal S}(\bigcup_{k=0}^\infty\cF_k)$,
so we may define,
for  all $A\in\bigcup_{k=0}^\infty{\mathcal S}(U_k)$, 
\begin{eqnarray}
\tau(A)=\lim_{l\to\infty}\tau_l(A), \label{def:taulimit}
\end{eqnarray}
the limit existing almost surely for all $A\in\bigcup_{k=0}^\infty{\mathcal S}(U_k)$ simultaneously.

Let $A, B\in \bigcup_{k=0}^\infty{\mathcal S}(U_k)$ be disjoint, so that $A, B\in {\mathcal S}(U_k)$ for some $k$.  Then, for $l\geq k$, $\tau_l(A\cup B)= \tau_l(A)+\tau_l(B)$.  Taking limits 
gives $\tau(A\cup B)=\tau(A)+\tau(B)$, so almost surely, $\tau$ is a finitely additive set function on 
$\bigcup_{k=0}^\infty{\mathcal S}(U_k)$.
Since $\{\tau_l(\cC^m), \cF_l\}_{l\geq 0}$ is a non-negative martingale, $\tau_l(\cC^m)<\infty$ almost surely.
By the extension theorems, see \cite{Taylor},
almost surely $\tau$ has a unique extension to  ${\mathcal S}\big(\bigcup_{k=0}^\infty{\mathcal S}(U_k)\big)$, i.e. $\tau$ is a random Borel
measure on $\cC^m$.

\begin{prop}\label{propmes}
The support of $\tau$ is contained in $E\cap\sigma(F)$, with $\tau(\cC^m)<\infty$ almost surely and $\tau(\cC^m)>0$ with positive probability. Moreover, for all $k \geq 0$ and $A\in U_k$,
\begin{equation}\label{tausquare}
\bE\big(\tau(A)^2\big)\leq c_0\mu(A) r^{k(\alpha+\beta+\log m / \log r )}.
\end{equation}
\end{prop}
\begin{proof}
Let $x\notin E\cap\sigma(F)$ but $x\in \cC^m$.  Since $\mu$ and $\nu$ have support on compact subsets of   $E$ and $F$ respectively, either $x\notin\text{supp}(\mu)$ or 
$\sigma^{-1}(x)\notin\text{supp}(\nu)$.  Without loss of generality, assume $x\notin\text{supp}(\mu)$.  Then there exists an open neighborhood of $x$ that does not intersect
$\text{supp}(\mu)$, which we may take to be an interval $A\in U_k$ for some $k$. Then by \eqref{martdef}, for all $l\geq k$, $\tau_l(A)=0$, so $\tau(A)=0$ and $x$ is not in the support of $\tau$.

Since  $\{\tau_l(\cC^m), \cF_l\}_{l\geq 0}$ is a non-negative martingale $0 \leq \tau (\cC^m)<\infty$ almost surely, and, since it is $\cL^2$-bounded, $\tau(\cC^m)>0$ with positive probability. Since $\cL^2$-bounded martingales converge in $\cL^2$, \eqref{tausquare} follows from \eqref{squarebound}.

\end{proof}

The $s$-energy of a measure $\upsilon$ is defined as $I_s(\upsilon)=\Int\Int\frac{\mathrm{d}\upsilon(x)\mathrm{d}\upsilon(y)}{d(x,y)^s}$.
We use the following variation of the potential theoretic method to bound the Hausdorff dimension of $E\cap\sigma(F)$, see
\cite[Section 4.3]{FractalGeometry}  and \cite[Chapter 8]{MatBook}.
\begin{thm}\label{thm:ptm}
Let $F$ be a Borel subset of $\cC^m$ and $\upsilon$ a measure with support in $F$ and $0< \upsilon(F)<\infty$. If $I_s(\upsilon)<\infty$,
then $\dim_H(F)\geq s$.
\end{thm}

To use this theorem, we find the expected value of $I_s(\tau)$, where $\tau$ is the random measure on $E\cap\sigma(F)$ constructed above.

\begin{lem}\label{lem:ptm}
Let $0<s< \alpha+\beta+\log m / \log r$. 
Then \[\bE\Big(\Int\Int\frac{\dt(x)\dt(y)}{d(x,y)^s}\Big)<\infty.\]
\end{lem}
\begin{proof}
For $x,y \in \cC^m$, we write $x\wedge y$ for the smallest interval $I$ such that $x,y\in I$.
We split the integral up into domains $\{x,y: x\wedge y \in I\}$ for each $I \in \cU$ and then use \eqref{tausquare}.

\begin{eqnarray*}
\bE\Big(\Int\Int\frac{\dt(x)\dt(y)}{d(x,y)^s}\Big) &\leq& 
\bE\Big(\sum_{k=0}^\infty\sum_{I\in U_k}\Int\Int_{x\wedge y=I}\frac{\dt(x)\dt(y)}{d(x,y)^s}\Big)\\
&\leq& \sum_{k=0}^\infty\sum_{I\in U_k}\bE\Big(r^{-ks}\Int\Int_{x\wedge y=I}\dt(x)\dt(y)\Big)\\
&\leq& \sum_{k=0}^\infty r^{-sk}\sum_{I\in U_k} \bE\big(\tau(I)^2\big)\\
&\leq&  c_0\sum_{k=0}^\infty r^{-sk}\sum_{I\in U_k}\mu(I) r^{k(\alpha+\beta+\log m / \log r )}\\
&\leq&  c_0\sum_{k=0}^\infty r^{k(\alpha+\beta+\log m / \log r -s)}\\
&<&\infty,
\end{eqnarray*}
since $\alpha+\beta+\log m / \log r -s>0$.
\end{proof}
Our final theorem now follows from the potential theoretic characterization of Hausdorff dimension.

\begin{thm}\label{hausineq}
Let $E$ and $F$ be Borel subsets of $\cC^m$. For all $\epsilon>0$,
\begin{equation}\label{epineq}
\dim_H(E\cap\sigma(F))>\dim_H E+\dim_H F+\frac{\log m }{\log r }-\epsilon 
\end{equation}
with positive probability.
\end{thm}
\begin{proof}
Let $0< \alpha <\dim_H E$, $0< \beta <\dim_H F$ and  $0 < s < \alpha+\beta+\log m / \log r$. From Lemma \ref{lem:ptm},  the $s$-energy of $\tau$, $I_s(\tau)$, is
finite almost surely. Provided that $\tau(\cC^m)>0$, which happens with positive probability by Proposition \ref{propmes}, then by Theorem \ref{thm:ptm}
\begin{eqnarray*}
\dim_H(E\cap\sigma(F))\geq s.
\end{eqnarray*}
By choosing $\alpha$ and $\beta$ sufficiently close to $\dim_H E$ and $\dim_H F$ and $s$ close to $\alpha+\beta+\log m / \log r$, we obtain \eqref{epineq} for any given $\epsilon>0$.
\end{proof}

We may rephrase Theorem  \ref{hausineq} as follows, with the case of equality coming from Theorem \ref{hausub}.

\begin{cor}\label{corol}
Let $E$ and $F$ be Borel subsets of $\cC^m$. Then
\[\mathrm{esssup}_{\sigma\in \mathrm{Iso }{\mathcal C}^m}\{\dim_H(E\cap\sigma(F))\}\geq \dim_H E+\dim_H F+\frac{\log m }{\log r }.\]
Equality holds if either $\dim_H E = {\overline \dim}_B\, E$ or  $\dim_H F = {\overline \dim}_B\, F$.
\end{cor}

It is natural to ask whether the lower bound  in Corollary \ref{corol} occurs with positive probability rather than just as an essential supremum. The following example shows that this is not true in general.

\begin{example}\label{eg}
For all $0<\alpha, \beta< -\log m/\log r$ with $\alpha + \beta+\log m/\log r> 0$ there exist Borel sets $E$ and $F$ in ${\mathcal C}^m$ such that $\dim_H E = \dim_B E = \alpha$ and $\dim_H F = \dim_B F = \beta$ and
$$\bP\bigg\{\dim_H(E\cap\sigma(F))\geq \dim_H E+\dim_H F+\frac{\log m }{\log r }\bigg\} = 0.$$
\end{example}
\begin{proof}
For each integer $i > 1/\alpha$, choose some interval $I_i \in U_i$ and construct a Borel set $E_i \subset I_i$ such that
$\dim_H E_i = \dim_B E_i = \alpha-1/i$. We may do this using a Cantor-type construction starting with $I_i$ but varying slightly the number of children intervals at each stage to get the required dimension. In doing so we may further ensure that $|U_k(E_i)| \leq r^{-k\alpha}=m^{-k\alpha \log r/\log m}$ for all $k\geq i$. 
Let $E= \bigcup_{i >1/\alpha} E_i$, so $\dim_H E = \alpha$.

In the same way, for $j> 1/\beta$, let $F= \bigcup_{j>1/\beta} F_j$, where $F_j\subset I_j$ for some $I_j \in U_j$ and $\dim_H F_j = \dim_B F_j = \beta-1/j$, with $|U_k(F_j)| \leq m^{-k\beta \log r/\log m}$ for all $k\geq j$. Thus 
$\dim_H F = \beta$. 

By Theorem \ref{thm2}  or  Theorem \ref{hausub}, for each $i > 1/\alpha, j> 1/\beta$,
$$\dim_H(E_i\cap\sigma(F_j))\leq \max\big\{\alpha +\beta + \log m/\log r - 1/i- 1/j, 0\big\}$$
with probability 1. Let $\epsilon >0$. Since $E \cap \sigma(F )= \bigcup_{i >1/\alpha}\bigcup_{j>1/\beta} E_i \cap \sigma(F_j)$,
\begin{align}
\bP\big(\dim_H &(E\cap\sigma(F))>\alpha +\beta + \log m/\log r - \epsilon\big)\nonumber\\
&\leq \sum_{1/i + 1/j  <\epsilon}\bP\big(\dim_H (E_i\cap\sigma(F_j))>\alpha +\beta + \log m/\log r - \epsilon\big)\nonumber\\
&\leq \sum_{1/i + 1/j  <\epsilon}\bP\big(E_i\cap\sigma(F_j)\neq \emptyset\big)\nonumber\\
&\leq \sum_{j\geq i >1/\epsilon}\bP\big(E_i\cap\sigma(F_j)\neq \emptyset\big)
+ \sum_{i\geq j >1/\epsilon}\bP\big(E_i\cap\sigma(F_j)\neq \emptyset\big)\label{sum}
\end{align}
For $j\geq i$, by construction $E_i$ is contained in at most 
$m^{-j\alpha \log r/\log m}$ intervals of $U_j$,
so 
$$\bP\big(E_i\cap\sigma(F_j)\neq \emptyset\big)\leq \bP\big(E_i\cap\sigma(I_j)\neq \emptyset\big)
\leq m^{-j\alpha \log r/\log m}\big/m^j=m^{-j(1+\alpha \log r/\log m)}.$$
Since $1+ \alpha \log r/\log m >0$, the left hand sum of \eqref{sum} is at most
$$  \sum_{i >1/\epsilon}\sum_{j \geq i}m^{-j(1+\alpha \log r/\log m)}
\leq c_1\sum_{i >1/\epsilon}m^{-i(1+\alpha \log r/\log m)}\leq c_2 m^{-(1+\alpha \log r/\log m)/\epsilon},$$
 where, provided that $\epsilon$ is sufficiently small, $c_1$ does not depend on $i$ and $\epsilon$ and $c_2$ does not depend on  $\epsilon$. With a similar estimate of $c_3m^{-(1+\beta\log r/\log m)/\epsilon}$ for the right hand sum of \eqref{sum} we conclude that
$$\lim _{\epsilon \to 0} \bP\big(\dim_H (E\cap\sigma(F))>\alpha +\beta + \log m/\log r - \epsilon\big)= 0.$$
\end{proof}

Nevertheless,  if  $E$ and $F$ are of positive Hausdorff measure in their dimensions the lower bound is attained with positive probability.

\begin{prop}\label{eg}
Let $E$ and $F$ be Borel subsets of $\cC^m$ and suppose that ${\mathcal H}^\alpha(E)>0$  and ${\mathcal H}^\beta(F)>0$ where $\alpha = \dim_H E$ and $\beta =\dim_H F$.
Then
\begin{equation}
\bP\bigg\{\dim_H(E\cap\sigma(F))\geq \dim_H E+\dim_H F+\frac{\log m }{\log r }\bigg\} > 0.\label{posprob}
\end{equation}
\end{prop}
\begin{proof}
In this case, the inequalities \eqref{frost} of Lemma \ref{lem:measurebounds} hold for suitable constants $c_E$ and $c_F$ with $\alpha$ and $\beta$ are actually equal to the dimensions of $E$ and $F$. The argument of Section 4 then goes through without the need to approximate these dimensions. The probability for which \eqref{epineq} holds is  just the probability that $\tau(\cC^m)>0$ which does not depend on $\epsilon>0$, so taking $\epsilon $ arbitraily small gives \eqref{posprob}.
\end{proof}

\end{section}


\begin{thebibliography}{99}

\bibitem{BGN}
L. Bartholdi, R. Grigorchuk and V. Nekrashevych.
\newblock From fractal groups to fractal sets. 
\newblock In {\em Fractals in Graz 2001}, pp 25--118, Trends Math., Birkh\"{a}user, Basel, 2003.

\bibitem{FractalGeometry}
K. Falconer.
\newblock {\em Fractal Geometry: Mathematical Foundations and Applications}.
\newblock John Wiley \& Sons, Hoboken, NJ, 3rd ed. 2014.

\bibitem{Kahane}
J.-P. Kahane.
\newblock Sur la dimension des intersections.
\newblock In {\em Aspects of mathematics and its applications}, volume~34 of
  {\em North-Holland Math. Library}, pp 419--430. North-Holland, Amsterdam,
  1986.

\bibitem{MatPaper}
P. Mattila.
\newblock Hausdorff dimension and capacities of intersections of sets in
  {$n$}-space.
\newblock {\em Acta Math.}, 152: 77--105, 1984.

\bibitem{Mat2}
P. Mattila.
\newblock  On the Hausdorff dimension and capacities of intersections.
\newblock {\em Mathematika}, 32: 213--217, 1986.

\bibitem{MatBook}
P. Mattila.
\newblock {\em Geometry of sets and measures in {E}uclidean spaces}.
\newblock Cambridge University Press, Cambridge, 1995.

\bibitem{Rogers}
C.~A. Rogers.
\newblock {\em Hausdorff Measures}.
\newblock Cambridge University Press, Cambridge, Reprint 1970.

\bibitem{Taylor}
S.~J. Taylor.
\newblock {\em Introduction to Measure and Integration}.
\newblock Cambridge University Press,
  Cambridge, 1998.

\end{thebibliography}
\end{document}